 \documentclass[authoryear,preprint,review,10pt]{elsarticle}

\usepackage{amssymb}

\newtheorem{theorem}{Theorem}

\newtheorem{definition}[theorem]{Definition}

\newtheorem{lemma}[theorem]{Lemma}

\newtheorem{remark}[theorem]{Remark}

\newenvironment{proof}[1][Proof]{\textbf{#1.} }{\ \rule{0.5em}{0.5em}}

\def \E{\mathbb{E}}

\begin{document}

\begin{frontmatter}



\title{ Neutral stochastic functional differential equation driven by
 fractional  Brownian motion and Poisson point processes}

\author[2]{Salah Hajji}
\ead{hajjisalahe@gmail.com}

\author[3]{El Hassan Lakhel}
\ead{e.lakhel@uca.ma}

\address[2]{ Department of Mathematics, Faculty of Sciences Semlalia, Cadi Ayyad University, 2390 Marrakesh, Morocco\\}
\address[3]{  National School of Applied Sciences, Cadi Ayyad University, 46000 Safi , Morocco}
\begin{abstract}
In this note we consider a class of neutral stochastic functional
differential equations  with finite delay driven simultaneously
by a fractional Brownian motion and a Poisson point processes  in a Hilbert space. We prove an existence and uniqueness result and  we establish some conditions
ensuring the exponential decay to zero in mean square for the mild
solution by means of the Banach fixed point principle.
\end{abstract}

\begin{keyword}
Mild solution; Semigroup of bounded linear operator; Fractional
powers of closed operators;  Fractional Brownian motion; Wiener
integral; Poisson point processes .



\MSC   60H15  \sep 60G15  \sep 60J65

\end{keyword}

\end{frontmatter}

\section{Introduction}
 In this paper, we study the existence, uniqueness and asymptotic behavior
  of mild solutions for a class of neutral functional stochastic differential equations with jumps described in the form:
\begin{equation}\label{eq1}
 \left\{\begin{array}{lll}
d[x(t)+g(t,x(t-r(t)))]&=&[Ax(t)+f(t,x(t-\rho(t))]dt+\sigma
(t)dB^H(t)\\&
+&\int_{\mathcal{U}}h(t,x(t-\theta(t)),y)\widetilde{N}(dt,dy),\;0\leq t \leq T,\nonumber\\
x(t)=\varphi(t) ,\;-\tau \leq t \leq 0.
\end{array}\right.
\end{equation}
where $\varphi\in \mathcal{D}:=D([-\tau,0],X$) the space of c\`adl\`ag functions
from $[-\tau,0]$  into $X$ equipped with the supremum norm
$|\varphi|_{\mathcal{D}}=sup_{s\in[-\tau,0]}\|\varphi(s)\|_X$ and  $A$ is the
infinitesimal generator of an analytic semigroup of bounded linear
operators, $(S(t))_{t\geq 0}$, in a Hilbert space $X$, $B^H$ is a
fractional Brownian motion on a real and separable Hilbert space
$Y$, $r,\; \rho,\; \theta:[0,+\infty)\rightarrow [0,\tau]\; (\tau
>0)$ are continuous and $f,g:[0,+\infty)\times X \rightarrow X,\;
\; \sigma:[0,+\infty) \rightarrow \mathcal{L}_2^0(Y,X)$,$\; \;$
$h:[0,+\infty)\times X\times \mathcal{U} \rightarrow X$ are appropriate
functions. Here $\mathcal{L}_2^0(Y,X)$ denotes the space of all
$Q$-Hilbert-Schmidt operators from $Y$ into $X$ (see section 2 below).\\
    We would like to mention that the theory for the stochastic differential equations (without
    delay) driven by a fractional Brownian motion (fBm) have recently been studied intensively
    (see e.g. \citet{Coutin}, \citet{Rascanu}, \citet{Ouknine}, \citet{Hu}, \cite{Saussereau} and the references therein).\\

    As, for the stochastic functional differential equations driven by a fBm, even much less has been
    done, as far as we know, there exist only a few papers published in this field. In \cite{Ferrante1}, the authors studied the existence and regularity
of the density  by using Skorohod integral based on the Malliavin calculus. \citet{Neuenkirch} studied the problem  by using rough path analysis.
     \cite{Ferrante2} studied the existence and convergence when the delay goes to zero by using the Riemann-Stieltjes integral.  Using also the Riemann-Stieltjes integral, \cite{boufoussi1} and \cite{boufoussi2} proved the existence and
    uniqueness of a mild solution and studied the dependence of the
    solution on the initial condition in finite and infinite dimensional space. Very recently,
     \cite{carab} have discussed the existence, uniqueness and exponential asymptotic behavior
     of mild solutions by using Wiener integral.\\

     By contrast, there has not been very much study of stochastic partial differential equations driven by jump processes, while these have
begun to gain attention recently. \citet{roc} showed by successive approximations the existence, uniqueness and large deviation
principle of stochastic evolution equations with jumps. In \citet{jin}, the exponential stability for neutral stochastic
partial differential  equation with a random perturbation of
the body forces in the form of Poisson and Brownian motion, were
studied. Under some circumstances, \citet{luo} established the
existence of strong solutions to stochastic partial functional
differential equations with Markovian switching and Poisson jumps,
meanwhile the moment exponential stability and almost sure
stability of mild solutions were investigated by the
Razumikhin Lyapunov type function argument and the comparison
principle, respectively.

On the other hand, to the best of our knowledge, there is no paper
which investigates the study of neutral stochastic functional differential
equations with delays driven both by fractional Brownian motion  and by  Poisson point processes. Thus, we will make the first attempt to study
such problem in this paper.\\
 Our results are inspired by the one in
\cite{boufoussi3} where the existence and uniqueness of mild
solutions to  model (\ref{eq1}) with $h=0$ is studied, as well as
some results on the  asymptotic behavior.\\

The  rest of this paper is organized as follows, In Section 2 we
introduce some notations, concepts, and basic results about
fractional Brownian motion, Poisson point processes,  Wiener
integral over Hilbert spaces and we recall some preliminary
results about analytic semi-groups and fractional power associated
to its generator. In section 3 by the Banach fixed point theorem
we consider a sufficient condition for the existence, uniqueness
and exponential decay  to zero in mean square  for   mild solutions
of  equation $(\ref{eq1})$.
\section{Preliminaries}

In this section, we collect some notions, conceptions and lemmas on Wiener integrals with respect to an infinite dimensional fractional Brownian
 and we recall some basic results about analytical semi-groups and  fractional powers of their infinitesimal generators, which will be used throughout the whole of this paper.\\
Let $(\Omega,\mathcal{F},\{\mathcal{F}_t\}_{t\geq0}, \mathbb{P})$
be a complete probability space satisfying the usual condition,
which means that the filtration is right continuous increasing
family and $\mathcal{F}_0$ contains all P-null sets.\\
Let $(\mathcal{U}, \mathcal{E}, \nu(du))$ be  a $\sigma$-finite measurable space. Given a stationary Poisson
point process $(p_t )_{t>0}$, which is defined on $(\Omega
,\mathcal{F}, P)$ with values in $\mathcal{U}$ and with
characteristic measure $\nu$ (see  \cite{ikeda}). We will denote by $N(t,du)$ be the counting measure of $p_t$ such that
$\widehat{N}(t, A) := \mathbb{E}(N(t, A)) = t\nu(A)$ for $A \in \mathcal{E}$.
Define $\widetilde{N}(t, du) := N(t, du) - t\nu(du)$, the Poisson
martingale measure generated by $p_t$.\\

Consider a time interval $[0,T]$ with arbitrary fixed horizon $T$
and let $\{\beta^H(t) , t \in [0, T ]\}$ the one-dimensional
fractional Brownian motion with Hurst parameter $H\in(1/2,1)$.
This means by definition that $\beta^H$ is a centered Gaussian
process with covariance function:
$$ R_H(s, t) =\frac{1}{2}(t^{2H} + s^{2H}-|t-s|^{2H}).$$
 Moreover $\beta^H$ has the following Wiener
integral representation:
\begin{equation}\label{rep}
\beta^H(t) =\int_0^tK_H(t,s)d\beta(s)
 \end{equation}
where $\beta = \{\beta(t) :\; t\in [0,T]\}$ is a Wiener process, and $K_H(t; s)$ is the kernel given by
$$K_H(t, s )=c_Hs^{\frac{1}{2}-H}\int_s^t (u-s)^{H-\frac{3}{2}}u^{H-\frac{1}{2}}du$$
for $t>s$, where $c_H=\sqrt{\frac{H(2H-1)}{\beta (2-2H,H-\frac{1}{2})}}$ and $\beta(,)$ denotes the Beta function. We put $K_H(t, s ) =0$ if $t\leq s$.\\
We will denote by $\mathcal{H}$ the reproducing kernel Hilbert space of the fBm. In fact $\mathcal{H}$ is the
closure of set of indicator functions $\{1_{[0;t]},  t\in[0,T]\}$ with respect to the scalar product
$$\langle 1_{[0,t]},1_{[0,s]}\rangle _{\mathcal{H}}=R_H(t , s).$$
The mapping $1_{[0,t]}\rightarrow \beta^H(t)$
 can be extended to an isometry between $\mathcal{H}$
and the first  Wiener chaos and
we will denote by $\beta^H(\varphi)$ the image of $\varphi$ by the previous isometry.

We recall that for $\psi,\varphi \in \mathcal{H}$ their scalar product in $\mathcal{H}$ is given by
$$\langle \psi,\varphi\rangle _{\mathcal{H}}=H(2H-1)\int_0^T\int_0^T\psi(s)\varphi(t)|t-s|^{2H-2}dsdt$$
Let us consider the operator $K_H^*$ from $\mathcal{H}$ to $L^2([0,T])$ defined by $$(K_H^*\varphi)(s)=\int_s^T\varphi(r)\frac{\partial K}{\partial r}(r,s)dr$$
We refer to \citet{nualart} for the proof of the fact that $K_H^*$ is an isometry between $\mathcal{H}$ and $L^2([0,T])$. Moreover for any $\varphi \in \mathcal{H}$, we have $$\beta^H(\varphi)=\int_0^T(K_H^*\varphi)(t)d\beta(t)$$
It follows from \citet{nualart} that the elements of $\mathcal{H}$ may be not functions but
distributions of negative order. In order to obtain a space of functions contained in $\mathcal{H}$, we consider the linear
space $|\mathcal{H}|$ generated by the measurable functions $\psi$ such that
$$\|\psi \|^2_{|\mathcal{H}|}:= \alpha_H  \int_0^T \int_0^T|\psi(s)||\psi(t)| |s-t|^{2H-2}dsdt<\infty$$
where $\alpha_H = H(2H-1)$. The space $|\mathcal{H}|$ is a Banach space with the norm  $\|\psi\|_{|\mathcal{H}|}$
and we have the following inclusions (see \citet{nualart})
\begin{lemma}\label{lem1}
$$\mathbb{L}^2([0,T])\subseteq \mathbb{L}^{1/H}([0,T])\subseteq |\mathcal{H}|\subseteq \mathcal{H}$$
and for any $\varphi\in \mathbb{L}^2([0,T])$, we have  $$\|\psi\|^2_{|\mathcal{H}|}\leq 2HT^{2H-1}\int_0^T |\psi(s)|^2ds$$
\end{lemma}
Let $X$ and $Y$ be two real, separable Hilbert spaces and let $\mathcal{L}(Y,X)$ be the space of bounded linear operator from $Y$ to $X$. For the sake of convenience, we shall use the same notation to denote the norms in $X,Y$ and $\mathcal{L}(Y,X)$.
Let $Q\in \mathcal{L}(Y,Y)$ be an operator defined by $Qe_n=\lambda_n e_n$ with finite trace
 $trQ=\sum_{n=1}^{\infty}\lambda_n<\infty$. where $\lambda_n \geq 0 \; (n=1,2...)$ are non-negative
  real numbers and $\{e_n\}\;(n=1,2...)$ is a complete orthonormal basis in $Y$.
 Let $B^H=(B^H(t))$, independent of the Poisson point process, be  $Y-$ valued fbm on
  $(\Omega,\mathcal{F}, \mathbb{P})$ with covariance $Q$ as
 $$B^H(t)=B^H_Q(t)=\sum_{n=1}^{\infty}\sqrt{\lambda_n}e_n\beta_n^H(t)$$
 where $\beta_n^H$ are real, independent fBm's. This process is  Gaussian, it
 starts from $0$, has zero mean and covariance:
 $$E\langle B^H(t),x\rangle\langle B^H(s),y\rangle=R(s,t)\langle Q(x),y\rangle \;\; \mbox{for all}\; x,y \in Y \;\mbox {and}\;  t,s \in [0,T]$$
In order to define Wiener integrals with respect to the $Q$-fBm, we introduce the space $\mathcal{L}_2^0:=\mathcal{L}_2^0(Y,X)$  of all $Q$-Hilbert-Schmidt operators $\psi:Y\rightarrow X$. We recall that $\psi \in \mathcal{L}(Y,X)$ is called a $Q$-Hilbert-Schmidt operator, if
$$  \|\psi\|_{\mathcal{L}_2^0}^2:=\sum_{n=1}^{\infty}\|\sqrt{\lambda_n}\psi e_n\|^2 <\infty$$
and that the space $\mathcal{L}_2^0$ equipped with the inner product
$\langle \varphi,\psi \rangle_{\mathcal{L}_2^0}=\sum_{n=1}^{\infty}\langle \varphi e_n,\psi e_n\rangle$ is a separable Hilbert space.\\

Now, let $\phi(s);\,s\in [0,T]$ be a function with values in $\mathcal{L}_2^0(Y,X)$, The Wiener integral of $\phi$ with respect to $B^H$ is defined by

\begin{equation}\label{int}
\int_0^t\phi(s)dB^H(s)=\sum_{n=1}^{\infty}\int_0^t \sqrt{\lambda_n}\phi(s)e_nd\beta^H_n(s)=\sum_{n=1}^{\infty}\int_0^t \sqrt{\lambda_n}(K_H^*(\phi e_n)(s)d\beta_n(s)
\end{equation}
where $\beta_n$ is the standard Brownian motion used to  present $\beta_n^H$ as in $(\ref{rep})$.\\
Now, we end this subsection by stating the following result which is fundamental to prove our result. It can be proved by  similar arguments as those used to prove   Lemma 2 in \cite{carab}.
\begin{lemma}\label{lem2}
If $\psi:[0,T]\rightarrow \mathcal{L}_2^0(Y,X)$ satisfies $\int_0^T \|\psi(s)\|^2_{\mathcal{L}_2^0}ds<\infty$
 then the above sum in $(\ref{int})$ is well defined as a $X$-valued random variable and
 we have$$ \mathbb{E}\|\int_0^t\psi(s)dB^H(s)\|^2\leq 2Ht^{2H-1}\int_0^t \|\psi(s)\|_{\mathcal{L}_2^0}^2ds$$
\end{lemma}
Now we turn to state some notations and basic facts about  the theory of semi-groups and fractional power operators.\\
Let $A:D(A)\rightarrow X$  be the infinitesimal generator of an
analytic semigroup, $(S (t))_{t\geq 0}$, of bounded linear
operators on $X$. For the theory of strongly continuous semigroup,
we refer to \cite{pazy}  and \cite{gold}.  We will
point out here some notations and properties that will be used in
this work. It is well known that there exist $M \geq 1$
and $\lambda \in \mathbb{R}$ such that $\|S (t)\|\leq M
e^{\lambda t}$ for every $t \geq 0$. \\
  If $(S (t))_{t\geq 0}$ is a
uniformly bounded and analytic semigroup such that $0 \in
\rho(A)$,  where $\rho (A)$ is the resolvent set of $A$, then it
is possible to define the fractional power $(-A)^{\alpha}$  for $0
< \alpha \leq 1$, as a closed linear operator on its domain
$D(-A)^{\alpha}$. Furthermore, the subspace $D(-A)^{\alpha}$ is
dense in $X$, and the expression
$$\|h\|_{\alpha} =\|(-A)^{\alpha}h\|$$ defines a norm in
$D(-A)^{\alpha}$. If $X_{\alpha}$ represents the space
$D(-A)^{\alpha}$ endowed with the norm $\|.\|_{\alpha}$, then the
following properties are well known (cf. \cite{pazy}, p. 74).
\begin{lemma}\label{lem3} Suppose that the preceeding conditions are satisfied.\\
(1) Let $0<\alpha \leq  1$. Then $X_{\alpha}$ is a Banach space.\\
(2) If $0 <\beta \leq \alpha $ then the injection $X_{\alpha}
\hookrightarrow X_{\beta}$ is continuous.\\
 (3) For every $0<\alpha \leq 1$ there exists $M_{\alpha} > 0 $ such that
 $$\| (-A)^{\alpha}S (t)\|\leq M_{\alpha}t^{-\alpha}e^{-\lambda t}
 , \;\;\;\;t>0,\;\; \lambda>0 .$$
 \end{lemma}

 \section{Main Results}

In this section, we consider existence, uniqueness  and
exponential stability of mild solution  to Equation (\ref{eq1}).
Our main method is the Banach fixed point principle. First we
define the space $\Pi$ of the c\`adl\`ag processes $x(t)$ as
follows:

\begin{definition}
Let the space $\Pi$ denote the set of all c\`adl\`ag processes
$x(t)$ such that  $x(t)=\varphi(t)\; t\in [-\tau,0]$ and  there
exist some constants $M^*=M^*(\varphi,a)>0$ and $a>0$
$$
\E\|x(t)\|^2\leq M^*e^{-at},\qquad\forall\; t\geq0.
$$
\end{definition}
\begin{definition}
$\|.\|_\Pi$ denotes the norm in $\Pi$ wich is defined by
$$
\|x\|_\Pi:= sup_{t\geq0}\E\|x(t)\|_X^2 \qquad for \;
 x\in \Pi.
$$
\end{definition}

\begin{remark}

It is routine to check that $\Pi$ is a Banach space  endowed with
the  norm $\|.\|_\Pi$.
 \end{remark}
 In order to obtain our main result,  we assume that the following conditions hold.
\begin{itemize}
\item [$(\mathcal{H}.1)$]
 $A$ is the infinitesimal generator of an analytic semigroup, $(S (t))_{t\geq
0}$, of bounded linear operators on $X$. Further, to avoid
unnecessary notations, we suppose that $0 \in \rho (A)$, and that,
see Lemma \ref{lem3},$$ \| S (t)\|\leq Me^{-\lambda t}\;
\;\;\mbox{and}\;\; \|(-A)^{1-\beta}S(t)\|\leq
\frac{M_{1-\beta}}{t^{1-\beta}} $$ for some constants
$M,\;\lambda\; ,\; M_{1-\beta}$ and every $t \in [0,T].$
\item [$(\mathcal{H}.2)$]There   exist positive constant $K_1>0$ such that, for all $t\in [0,T]$ and $x,y\in X $
  $$
  \|f(t,x)-f(t,y)\|^2\leq K_1 \|x-y\|^2.
  $$
\item [$(\mathcal{H}.3)$]There exist constants $0 <
\beta < 1,\; K_2>0$ such
 that the function $g$ is $X_{\beta}$-valued and satisfies for all  $t\in [0,T]$ and $x,y\in  X$
 $$
 \|(-A)^{\beta}g(t,x)-(-A)^{\beta}g(t,y)\|^2 \leq K_2 \|x-y\|^2.
 $$
\item [$(\mathcal{H}.4)$]The function $(-A)^{\beta}g$ is continuous in the quadratic mean sense:
$$\mbox{For all}\;\; x\in D([0,T], \mathbb{L}^2(\Omega,X)),\;\;\lim_{t\rightarrow s}\mathbb{E}\|(-A)^{\beta}g(t,x(t))-(-A)^{\beta}g(s,x(s))\|^2=0.$$
\item [$(\mathcal{H}.5)$] There exists some $\gamma>0$ such that the  function $\sigma:[0,+\infty)\rightarrow \mathcal{L}_2^0(Y,X)$
satisfies
 $$\int_0^T e^{2\gamma s}\|\sigma(s)\|^2_{\mathcal{L}_2^0}ds< \infty,\;\; \forall\; T>0. $$
 \item [$(\mathcal{H}.6)$] There   exist positive constant $K_3>0$ such that, for all $t\in [0,T]$ and $x,y\in X $
 $$
\int_{\mathcal{U}}\|h(t,x,z)-h(t,y,z)\|_X^2\nu(dz)\leq K_3\|x-y\|_X^2.
 $$
 We further assume that $g(t,0)=f(t,0)=h(t,0,z)=0$ for all
 $t\geq0$ and $z\in \mathcal{U}$. Moreover, we suppose  that $\E|\varphi|_\mathcal{D}^2 <\infty$.
\end{itemize}

Similar to the deterministic situation we give the following definition of mild solutions
for equation (\ref{eq1}).
\begin{definition}
A $X$-valued  process $\{x(t),\;t\in[-\tau,T]\}$, is called  a mild
solution of equation (\ref{eq1}) if
\begin{itemize}
\item[$i)$] $x(.)$ has c\`adl\`ag path, and
$\int_{0}^T\|x(t)\|^2dt<\infty$ almost surely;
\item[$ii)$] $x(t)=\varphi(t), \, -\tau \leq t \leq 0$.
\item[$iii)$]For arbitrary $t \in [0,T]$, we have

$$
\left\{\begin{array}{ll}
x(t)&=S(t)(\varphi(0)+g(0,\varphi(-r(0))))-g(t,x(t-r(t)))\\
&- \int_0^t AS(t-s)g(s,x(s-r(s)))ds +\int_0^t S(t-s)f(s,x(s-\rho (s))ds\\
&+\int_0^t S(t-s)\sigma(s)dB^H(s)\\
&+\int_0^t \int_ {\mathcal{U}}
S(t-s)h(s,x(s-\theta(s)),y)\widetilde{N}(ds,dy)\;\;\;
\mathbb{P}-a.s
\end{array}\right.
$$
\end{itemize}
\end{definition}
We can now state the main result of this paper.
\begin{theorem}\label{th1}
Suppose that $(\mathcal{H}.1)-(\mathcal{H}.6)$ hold and that
 $$
4(K_2\|(-A)^{-\beta}\|^2+K_2
 M_{1-\beta}^2\lambda^{-2\beta}\Gamma(\beta)^2+K_1 M^2 \lambda^{-2}+M^2
 K_3(2\lambda)^{-1})<1
$$
where $\Gamma(.)$ is the Gamma function, $M_{1-\alpha}$ is the
corresponding constant in Lemme \ref{lem3}. \\
If the initial value $\varphi(t)$ satisfies
$$
\E\|\varphi(t)\|^2\leq M_0 \E|\varphi|_{\mathcal{D}}^2 e^{-at},\quad
t\in[-\tau,0],
$$
for some $M_0>0$, $a>0$; then, for all $T>0$, the equation
(\ref{eq1}) has a unique mild solution on $[-\tau,T]$ and is
exponential decay to zero in mean square, i.e., there exists a
pair of positive constants $a>0$ and $M^*=M^*(\varphi,a)$ such that
$$
\E\|x(t)\|^2\leq M^*e^{-at}, \; \forall t\geq0.
$$
\end{theorem}
\begin{proof}
Define the mapping $\Psi$ on $\Pi$ as follows:
$$
\Psi(x)(t):=\varphi(t), \qquad t\in [-\tau,0],
$$
 and for $t\in [0,T]$
 $$
 \begin{array}{ll}
 \Psi(x)(t)&=S(t)(\varphi(0)+g(0,\varphi(-r(0))))-g(t,x(t-r(t)))-\int_0^t AS(t-s)g(s,x(s-r(s)))ds\\
&+\int_0^t S(t-s)f(s-\rho (s))ds+\int_0^tS(t-s)\sigma(s)dB^H(s)\\
&+\int_0^t \int_ {\mathcal{U}}
S(t-s)h(s,x(s-\theta(s)),y)\widetilde{N}(ds,dy).
 \end{array}
 $$
 Then it is clear that to prove the existence of
mild solutions to equation (\ref{eq1}) is equivalent to find a fixed
point for the operator $\Psi$.\\

We will show by using Banach
fixed point theorem that $\Psi$ has a unique fixed point. First we
show that $\Psi(\Pi)\subset \Pi$.\\

Let $x(t)\in \Pi$, then we have
\begin{eqnarray}\label{eqa}
  \E\|\Psi(x)(t)\|^2&\leq & 6 \E \|S(t)(\varphi(0)+g(0,\varphi(-r(0))))\|^2\nonumber\\
 &+&6 \E \|g(t,x(t-r(t)))\|^2+6 \E \|
 \int_0^t AS(t-s)g(s,x(s-r(s)))ds\|^2\nonumber\\
&+&6 \E \|\int_0^t S(t-s)f(s-\rho (s))ds\|^2+6 \E \|\int_0^tS(t-s)\sigma(s)dB^H(s)\|^2\nonumber\\
&+&6 \E \|\int_0^t \int_ {\mathcal{U}}
S(t-s)h(s,x(s-\theta(s)),y)\widetilde{N}(ds,dy)\|^2\nonumber\\
&:=&6(I_1+I_2+I_3+I_4+I_5+I_6).
 \end{eqnarray}
Now, let us estimate the terms on the right of the inequality (\ref{eqa}).\\
Let $M^*=M^*(\varphi,a)>0$ and $a>0$ such that
$$ \E\|x(t)\|^2\leq M^*e^{-at},\;\;\; \forall\; t\geq 0. $$
Without loss of generality we
may assume that $0<a<\lambda$. Then, by assumption $(\mathcal{H}.1)$ we
have
\begin{equation}\label{eqa1}
I_1\leq M^2\E\|\varphi(0)+g(0,\varphi(-r(0)))\|^2e^{-\lambda
t}\leq C_1 e^{-\lambda t}
\end{equation}
where $ C_1=M^2\E\|\varphi(0)+g(0,\varphi(-r(0)))\|^2<+\infty$.\\
 By using assumption $(\mathcal{H}.3)$  and the fact that the operator $(-A)^{-\beta}$ is bounded, we obtain that

 \begin{eqnarray}\label{eqa2}
 I_2&\leq & \|(-A)^{-\beta}\|^2\mathbb{E}\|(-A)^{\beta}g(t,x(t-r(t)))-(-A)^{\beta}g(t,0)\|^2\nonumber\\
 &\leq & K_2\|(-A)^{-\beta}\|^2\mathbb{E}\|x(t-r(t))\|^2\nonumber\\
 &\leq & K_2\|(-A)^{-\beta}\|^2(M^*e^{-(t-r(t))}+\E\|\varphi(t-r(t))\|^2)\nonumber\\
 &\leq &
 K_2\|(-A)^{-\beta}\|^2(M^*+M_0\E|\varphi|_{\mathcal{D}}^2)e^{-(t-r(t))}\nonumber\\
 &\leq &
 K_2\|(-A)^{-\beta}|^2(M^*+M_0\E|\varphi|_{\mathcal{D}}^2)e^{-at}e^{a\tau}\nonumber\\
 &\leq & C_2 e^{-at}
 \end{eqnarray}
 where $  C_2=K_2\|(-A)^{-\beta}\|^2(M^*+M_0\E|\varphi|_{\mathcal{D}}^2)e^{a\tau}<+\infty$.\\
To estimate $I_3$, we use the trivial identity

\begin{equation}\label{gam}
c^{-\alpha}=\frac{1}{\Gamma(\alpha)}\int_0^{+\infty}t^{\alpha-1}e^{-ct}\;
\; ,\;\forall \; c>0.
\end{equation}
Using H\"older's inequality, Lemma \ref{lem3} together with  assumption
$(\mathcal{H}.3)$  and the  identity  (\ref{gam}), we get
 \begin{eqnarray}\label{eqa3}
I_3&\leq &\E \|
 \int_0^t AS(t-s)g(s,x(s-r(s)))ds\|^2\nonumber\\
 &\leq &\int_0^t \|(-A)^{1-\beta}S(t-s)\|ds\int_0^t\| (-A)^{1-\beta}S(t-s)\|\E\|(-A)^{\beta}g(s,x(s-r(s)))\|^2ds\nonumber\\
&\leq &
M_{1-\beta}^2K_2\int_0^t(t-s)^{\beta-1}e^{-\lambda(t-s)}ds
\int_0^t(t-s)^{\beta-1}e^{-\lambda(t-s)}\E\|x(s-r(s))\|^2ds\nonumber\\
&\leq &
M_{1-\beta}^2K_2\lambda^{-\beta}\Gamma(\beta)
\int_0^t(t-s)^{\beta-1}e^{-\lambda(t-s)}(M^*
+M_0\E|\varphi|_\mathcal{D}^2)e^{-as}e^{a\tau}ds\nonumber\\
&\leq &
M_{1-\beta}^2K_2\lambda^{-\beta}\Gamma(\beta)(M^*+M_0\E|\varphi|_\mathcal{D}^2)
e^{-at}e^{a\tau}\int_0^t(t-s)^{\beta-1}e^{(a-\lambda)(t-s)}ds\nonumber\\
&\leq &
M_{1-\beta}^2K_2\lambda^{-\beta}\Gamma^2(\beta)(\lambda-a)^{-1}
(M^*+M_0\E|\varphi|_\mathcal{D}^2)e^{a\tau}e^{-at}\nonumber\\
&\leq & C_3 e^{-at}
\end{eqnarray}

where $  C_3=M_{1-\beta}^2K_2\lambda^{-\beta}\Gamma^2(\beta)
(\lambda-a)^{-1}(M^*+M_0\E|\varphi|_D^2)e^{a\tau}<+\infty$.\\

Similar computations can be used to estimate the term $I_4$.
\begin{eqnarray}\label{eqa4}
I_4&\leq &\E \|\int_0^t S(t-s)f(s,x(s-\rho(s)))ds\|^2\nonumber\\
&\leq &M^2
K_1\int_0^te^{-\lambda(t-s)}ds\int_0^te^{-\lambda(t-s)}\E\|x(s-\rho(s))\|^2ds\nonumber\\
&\leq &M^2K_1\lambda^{-1}\int_0^te^{-\lambda(t-s)}(M^*+M_0
\E|\varphi|_{\mathcal{D}}^2)e^{-as}e^{a\tau}ds\nonumber\\
&\leq &M^2K_1\lambda^{-1}(M^*+M_0\E|\varphi|_{\mathcal{D}}^2)e^{-at}
e^{a\tau}\int_0^te^{(a-\lambda)(t-s)}ds\nonumber\\
&\leq &M^2K_1\lambda^{-1}(\lambda-a)^{-1}(M^*+M_0\E
|\varphi|_{\mathcal{D}}^2)e^{-at}e^{a\tau}\nonumber\\
&\leq & C_4 e^{-at}
\end{eqnarray}

By using Lemma  \ref{lem2}, we get that
\begin{eqnarray}\label{eqa5}
I_5&\leq & \E\|\int_0^t S(t-s)\sigma(s)dB^H(s)\|^2\nonumber\\
&\leq & 2M^2Ht^{2H-1}\int_0^t
e^{-2\lambda(t-s)}\|\sigma(s)\|_{\mathcal{L}_2^0}^2ds,
\end{eqnarray}
If $\gamma<\lambda $,  then the following estimate holds
\begin{eqnarray}\label{eqa51}
I_5 &\leq & 2M^2Ht^{2H-1}\int_0^t
e^{-2\lambda(t-s)}e^{-2\gamma(t-s)}e^{2\gamma(t-s)}\|\sigma(s)\|_{\mathcal{L}_2^0}^2ds\nonumber\\
&\leq & 2M^2Ht^{2H-1}e^{-2\gamma t}\int_0^t
e^{-2(\lambda-\gamma)(t-s)}e^{2\gamma s
}\|\sigma(s)\|_{\mathcal{L}_2^0}^2ds\nonumber\\
&\leq & 2M^2HT^{2H-1}e^{-2\gamma t}\int_0^T
e^{-2(\lambda-\gamma)(t-s)}e^{2\gamma s
}\|\sigma(s)\|_{\mathcal{L}_2^0}^2ds\nonumber\\
&\leq & 2M^2HT^{2H-1}e^{-2\gamma t}\int_0^T e^{2\gamma s
}\|\sigma(s)\|_{\mathcal{L}_2^0}^2ds,
\end{eqnarray}

If $\gamma>\lambda $,  then the following estimate holds
\begin{equation}\label{eqa52}
I_5\leq 2M^2HT^{2H-1}e^{-2\lambda t}\int_0^T e^{2\gamma s}\|\sigma(s)\|_{\mathcal{L}_2^0}^2ds
\end{equation}
In virtue of $(\ref{eqa5}), (\ref{eqa51})$ and $(\ref{eqa52})$ we obtain
\begin{equation}\label{eqa50}
I_5\leq C_5 e^{-min(\lambda,\gamma)t}
\end{equation}
where $  C_5= 2M^2HT^{2H-1}\int_0^T e^{2\gamma s}\|\sigma(s)\|_{\mathcal{L}_2^0}^2ds<+\infty$.\\
On the other hand, by assumptions  $(\mathcal{H}.1)$  and
$(\mathcal{H}.6)$, we get

\begin{eqnarray}\label{eqa6}
I_6 &\leq &\E \|\int_0^t \int_ {\mathcal{U}}
S(t-s)h(s,x(s-\theta(s)),y)\widetilde{N}(ds,dy)\|^2\nonumber\\
&\leq &M^2 \E\int_0^te^{-2\lambda(t-s)} \int_ {\mathcal{U}}
\|h(s,x(s-\theta(s)),y)\|^2\nu(dy)ds\nonumber\\
&\leq &M^2 K_3 \int_0^te^{-2\lambda(t-s)}
\E\|x(s-\theta(s)\|^2ds\nonumber\\
&\leq &M^2 K_3 \int_0^te^{-2\lambda(t-s)}
(M^*+M_0\E|\varphi|_{\mathcal{D}}^2)e^{-as}e^{a\tau}ds\nonumber\\
&\leq &M^2 K_3
(M^*+M_0\E|\varphi|_{\mathcal{D}}^2)e^{-at}e^{a\tau}\int_0^te^{(-2\lambda+a)(t-s)}ds\nonumber\\
&\leq &M^2 K_3 (M^*+M_0\E|\varphi|_{\mathcal{D}}^2)e^{a\tau}
(2\lambda-a)^{-1}e^{-at}\nonumber\\
&\leq &  C_6 e^{-at},
\end{eqnarray}
where $ C_6= M^2 K_3 (M^*+M_0\E|\varphi|_{\mathcal{D}}^2)e^{a\tau}
(2\lambda-a)^{-1}<+\infty.$\\
 Inequalities $(\ref{eqa1}), (\ref{eqa2}),(\ref{eqa3}), (\ref{eqa4}),(\ref{eqa50})$ and $(\ref{eqa6})$ together imply taht

$$
\E\|\Psi(x)(t)\|^2\leq \overline{M}e^{-\overline{a} t}\;,\;\;
t\geq0.
$$
for some $\overline{M}> 0$ and $\overline{a}>0$.\\
Next we showt hat $\Psi(x)(t)$ is c\`adl\`ag process on $\Pi$. Let
$0 <t<T$  and $h>0$  be sufficiently small. Then for any fixed
$x(t)\in \Pi$, we have
\begin{eqnarray*}
&\E\|&\Psi(x)(t+h)-\Psi(x)(t)\|^2\\
&\leq & 6\E\|(S(t+h)-S(t))(\varphi(0)+g(0,\varphi(-r(0))))\|^2\\
&+&6\E\|g(t+h,x(t+h-r(t+h)))-g(t,x(t-r(t)))\|^2\\
&+&6\E\|\int_0^{t+h} AS(t+h-s)g(s,x(s-r(s))ds -\int_0^t AS(t-s)g(s,x(s-r(s))ds\|^2\\
&+&6\E\|\int_0^{t+h} S(t+h-s)f(s-\rho (s))ds-\int_0^t S(t-s)f(s-\rho (s))ds\|^2\\
&+&6\E\|\int_0^{t+h}S(t+h-s)\sigma(s)dB^H(s)-\int_0^tS(t-s)\sigma(s)dB^H(s)\|^2\\
&+&6\E \|\int_0^{t+h} \int_ {\mathcal{U}}
S(t+h-s)h(s,x(s-\theta(s)),y)\widetilde{N}(ds,dy)\\
&-&\int_0^{t} \int_ {\mathcal{U}}
S(t-s)h(s,x(s-\theta(s)),y)\widetilde{N}(ds,dy)\|^2\\
 &=& 6 \sum_{1\leq i \leq 6}I_i(h).
\end{eqnarray*}
For $i=1,2,...,5,$  the
 terms $I_i(h)$ can be treated in the same way as in the proof  of Theorem 5
 in \citet{boufoussi3}.\\ For the term $I_6(h)$, we have  by assumption $(\mathcal{H}.1)$

\begin{eqnarray}\label{b}
I_6(h) &\leq &2\E \|\int_0^t \int_ {\mathcal{U}}
(S(t+h-s)-S(t-s))h(s,x(s-\theta(s)),y)\widetilde{N}(ds,dy)\|^2\nonumber\\
&+ &2\E \|\int_t^{t+h} \int_ {\mathcal{U}}
S(t+h-s)h(s,x(s-\theta(s)),y)\widetilde{N}(ds,dy)\|^2\nonumber\\
&\leq &2M^2\|S(h)-I\|\E \int_0^t \int_ {\mathcal{U}} e^{-2\lambda(t-s)}
 \|h(s,x(s-\theta(s)),y)\|^2\nu(dy)ds\nonumber\\
&+ &2 M^2\E \int_t^{t+h} \int_ {\mathcal{U}} e^{-2\lambda(t+h-s)}
\|h(s,x(s-\theta(s)),y)\|^2\nu(dy)ds
\end{eqnarray}
By assumption $(\mathcal{H}.6)$, we have
\begin{eqnarray}\label{b1}
 \E\int_0^t \int_ {\mathcal{U}} e^{-2\lambda(t-s)}
\|h(s,x(s-\theta(s)),y)\|^2\nu(dy)ds
&\leq&K_3\int_0^te^{-2\lambda(t-s)}
\E\|x(s-\theta(s)\|^2ds\nonumber\\
&\leq&K_3\int_0^te^{-2\lambda(t-s)}
(M^*+M_0\E|\varphi|_{\mathcal{D}}^2)e^{-as}e^{a\tau}ds\nonumber\\
&\leq & K_3 (M^*+M_0\E|\varphi|_{\mathcal{D}}^2)e^{a\tau}(2\lambda-a)^{-1}e^{-at}
\end{eqnarray}

Inequality $(\ref{b1})$ imply that there exist a constant $B>0$ such that

\begin{equation}\label{b2}
 \E\int_0^t \int_ {\mathcal{U}}
e^{-2\lambda(t-s)} \|h(s,x(s-\theta(s)),y)\|^2\nu(dy)ds\leq B
\end{equation}
Using the  strong continuity of $S(t)$ together with inequalities $(\ref{b})$ and $(\ref{b2})$ we obtain that  $I_6(h)\rightarrow0$ as $h\rightarrow0$.

The above arguments show that $\Psi(x)(t)$ is c\`adl\`ag process.
Then, we conclude that $\Psi(\Pi)\subset \Pi$.\\
Now, we are going to show that $\Psi:\Pi\rightarrow\Pi$ is a
contraction mapping.  For this end, fix $x,y\in \Pi$,  we have
\begin{eqnarray}\label{c}
\E\|\Psi(x)(t)&-&\Psi(y)(t)\|^2\nonumber\\
&\leq& 4\E\|g(t,x(t-r(t)))-g(t,y(t-r(t)))\|^2\nonumber\\
&&+ 4\E\|\int_0^tAS(t-s)(g(s,x(s-r(s)))-g(s,y(s-r(s)))ds\|^2\nonumber\\
&&+ 4\E\|\int_0^tS(t-s)(f(s,x(s-\rho(s)))-f(s,y(s-\rho(s)))ds\|^2\nonumber\\
&&+ 4\E\int_0^tS(t-s) (\int_ {\mathcal{U}}
h(s,x(s-\theta(s)),z)-h(s,y(s-\theta(s)),z))\widetilde{N}(ds,dz)\|^2\nonumber\\
&&:= 4(J_1+J_2+J_3+J_4).
\end{eqnarray}
We estimate the various terms of the right hand of $(\ref{c})$ separately.\\
For the first term, we have

\begin{eqnarray}\label{c1}
J_1&\leq&\E\|g(t,x(t-r(t)))-g(t,y(t-r(t)))\|^2\nonumber\\
&\leq& K_2\|(-A)^{-\beta}\|^2\E\|x(s-r(s))-y(s-r(s))\|^2\nonumber\\
&\leq&K_2\|(-A)^{-\beta}\|^2\sup_{s\geq0} \E\|x(s)-y(s)\|^2.
\end{eqnarray}

For the second term, combing Lemma \ref{lem3} and H\"older's
inequality, we get

\begin{eqnarray}\label{c2}
J_2&\leq&\E\|\int_0^tAS(t-s)(g(s,x(s-r(s)))-g(s,y(s-r(s)))ds\|^2\nonumber\\
 &\leq&K_2
 M_{1-\beta}^2\int_0^t(t-s)^{\beta-1}e^{-\lambda(t-s)}ds
 \int_0^t(t-s)^{\beta-1}e^{-\lambda(t-s)}\E\|x(s-r(s))-y(s-r(s))\|^2ds\nonumber\\
&\leq&K_2
 M_{1-\beta}^2\lambda^{-\beta}\Gamma(\beta)\int_0^t(t-s)^{\beta-1}e^{-\lambda(t-s)} ds (\sup_{s\geq0} \E\|x(s)-y(s)\|^2)\nonumber\\
&\leq&K_2
 M_{1-\beta}^2\lambda^{-2\beta}\Gamma(\beta)^2 \sup_{s\geq0} \E\|x(s)-y(s)\|^2
\end{eqnarray}

For the third term, by assumption $(\mathcal{H}.2)$, we get that

\begin{eqnarray}\label{c3}
J_3&\leq&\E\|\int_0^tS(t-s)(f(s,x(s-\rho(s)))-f(s,y(s-\rho(s)))ds\|^2\nonumber\\
&\leq& K_1 M^2\int_0^te^{-\lambda (t-s)}ds\int_0^te^{-\lambda
(t-s)}\E\|x(s-\rho(s))-y(s-\rho(s))\|^2ds\nonumber\\
&\leq&K_1 M^2 \lambda^{-2}\sup_{s\geq0} \E\|x(s)-y(s)\|^2
\end{eqnarray}
For the last term, by using assumption $(\mathcal{H}.6)$, we get

\begin{eqnarray}\label{c4}
J_4&\leq&\E\int_0^tS(t-s) (\int_ {\mathcal{U}}
h(s,x(s-\theta(s)),z)-h(s,y(s-\theta(s)),z))\widetilde{N}(ds,dz)\|^2\nonumber\\
&\leq& M^2\E\int_0^te^{-2\lambda(t-s)} (\int_ {\mathcal{U}}
\|h(s,x(s-\theta(s)),z)-h(s,y(s-\theta(s)),z))\|^2\nu(dz)ds\nonumber\\
&\leq& M^2 K_3(2\lambda)^{-1} \sup_{s\geq0} \E\|x(s)-y(s)\|^2.
\end{eqnarray}
Thus, inequality $(\ref{c1}), (\ref{c2}), (\ref{c3})$  and  $(\ref{c4})$ together imply

\begin{eqnarray*}
\sup_{t\geq 0}\E\|\Psi(x)(t)-\Psi(y)(t)\|^2\leq
&4&(K_2\|(-A)^{-\beta}\|^2+K_2
 M_{1-\beta}^2\lambda^{-2\beta}\Gamma(\beta)^2+K_1 M^2 \lambda^{-2}\\&&+M^2
 K_3(2\lambda)^{-1})
(\sup_{t\geq0} \E\|x(t)-y(t)\|^2).
\end{eqnarray*}

Therefore by the condition of the theorem it follows that $\Psi$
is a contractive mapping. Thus by the Banach fixed point theorem
$\Psi$ has the fixed point $x(t)\in\Pi$, which is a unique mild
solution to  (\ref{eq1}) satisfying $x(s)=\varphi(s)$ on
$[-\tau,0]$. \\
By the definition of the space $\Pi$ this solution is
exponentially stable in mean square.  This completes the proof.
 \end{proof}


\end{document}